\documentclass[10pt]{amsart}   
\usepackage{amssymb,amscd,latexsym}   
\usepackage{amsmath,amsbsy}
\usepackage{amsthm}
\usepackage{mathdots}
\usepackage[pagebackref]{hyperref}
\usepackage{color}
\usepackage[all]{xy}
\usepackage{tikz-cd}
\usepackage[english]{babel}
\usepackage{thmtools}
\newtheorem{Theorem}{Theorem}[section] 
\newtheorem{Corollary}[Theorem]{Corollary} 
 
\newtheorem{Proposition}[Theorem]{Proposition}
 \theoremstyle{definition}
 \newtheorem{Definition}[Theorem]{Definition}
  \newtheorem{Example}[Theorem]{Example} 
   \newtheorem{Remark}[Theorem]{Remark}  
   \newtheorem{Question}[Theorem]{Question}

      \numberwithin{equation}{section}
        
        \newtheorem{Notation}[Theorem]{Notation}
  
\DeclareMathOperator{\Bir}{Bir}
 
\DeclareMathOperator{\Image}{Im}

\DeclareMathOperator{\Proj}{Proj} 
\DeclareMathOperator{\Spec}{Spec}

 \DeclareMathOperator{\Dim}{dim}
 \DeclareMathOperator{\edim}{edim}
 \DeclareMathOperator{\HF}{HF}
  \DeclareMathOperator{\codim}{codim}
  
\DeclareMathOperator{\depth}{depth}
 \DeclareMathOperator{\Ht}{ht}
 
\DeclareMathOperator{\grade}{grade}
 
 \DeclareMathOperator{\A}{\boldsymbol{\alpha}}

\newcommand{\binomial}[2]{{#1 \choose #2}}

\newcommand{\fm}{\mathfrak{m}}

\newcommand{\fp}{\mathfrak{p}}
\newcommand{\fq}{\mathfrak{q}}
\newcommand{\fa}{\mathfrak{a}}
\newcommand{\fb}{\mathfrak{b}}

\newcommand{\fn}{\mathfrak{n}}

\def\phi{\varphi}

\def\aa{{\bf a}}
\def\bb{{\bf b}}

\def\xx{{\bf x}}\def\yy{{\bf y}} \def\aa{{\bf a}}
  \def\XX{{\bf X}}\def\YY{{\bf Y}}
\def\ZZ{{\bf Z}}

\newcommand{\llar}{-\kern-5pt-\kern-5pt\longrightarrow}
\def\restr{{\kern-1pt\restriction\kern-1pt}}

\begin{document}
\title{Open loci of ideals with applications to Birational maps}
\author{S. Hamid Hassanzadeh and  Maral Mostafazadehfard }
\email{hamid@im.ufrj.br}\email{maral@im.ufrj.br}
\address{Departamento de Matem\'atica, Centro de Tecnologia, Cidade Universit\'aria da Universidade Federal do Rio de Janeiro, 21941-909 Rio de Janeiro, RJ, Brazil}
\dedicatory{Dedicated to Aron Simis on the occasion of his eightieth birthday.}
\date{\today}


\maketitle
\begin{abstract} In this work we show that  the  loci of ideals in principal class, ideals of grade at least two, and ideals of maximal analytic spread are Zariski open sets in the parameter space.  As an application, we show that the set of birational maps of {\it clear polynomial degree} $d$ over an arbitrary projective variety $X$, denoted by $\Bir(X)_{d}$, is a constructible set. This extends a previous result by Blanc and Furter.
\end{abstract}

\section{Introduction}

In classical commutative ring theory, an open locus of a ring $R$ refers to a set of prime ideals $\fp\in \Spec(R)$ such that $R_{\fp}$ satisfies a particular property. However, in subjects such as Hilbert schemes, the focus is on sets of ideals (in a polynomial ring) that satisfy specific properties, such as having the same Hilbert function. The term "open loci of ideals" refers to the open Zariski subsets of a parametric space whose points correspond to ideals with such properties.

In this paper, we study three properties of homogeneous ideals in a positively graded standard Noetherian domain over a field. The first property, described in \autoref{Elimination}, is belonging to the principal class. The principal class refers to the class of ideals whose codimension is equal to the number of their generators. The second property, presented in  \autoref{LG2}, is having grade at least 2. The grade of an ideal is the maximum length of a regular sequence inside the ideal. Finally, we investigate the Noether Normalization locus in  \autoref{LNn+1}. This locus is determined by ideals with maximum analytic spread, which is defined as the dimension of the special fiber of the ideal over the base ring. We show all of these loci are open.


Despite the interesting algebraic nature of the aforementioned open loci, we have discovered a significant application of them in the study of the group of birational maps over an arbitrary projective variety. In this context, it is worth mentioning some historical background.

Let $k$ be a field and $X$ be an algebraic variety over $k$. The group of birational transformations on $X$ is denoted by $\Bir_k(X)$ or simply $\Bir(X)$. Although this group has been the subject of research for over 160 years, many mysteries about it remain, even in the case where $X=\mathbb{P}^n$.

Regarding the group theoretic structures, there are several classical results for $\Bir(\mathbb{P}^2)$, dating back to Noether and Castelnuovo. However, it was only in 2013 that Cantat and Lamy discovered the non-simplicity and some unusual behaviors of normal subgroups in this group \cite{CL13}. The case of $\Bir(\mathbb{P}^n)$ is more intricate, and less is known about the group structure of $\Bir(X)$ when $X$ is not of general type.  

In 2010, J.-P. Serre \cite{Se10} introduced a topology on $\Bir(X)$ by considering morphisms from algebraic varieties. In 2013, Blanc and Furter proved that there is no structure of an algebraic variety on $\Bir(\mathbb{P}^n)$ for $n \geq 2$, but on the positive side, they showed that for any fixed integer $d$, the set of birational transformations of degree $d$, denoted by $\Bir(\mathbb{P}^n)_d$, has the structure of an algebraic variety. This led to the classification of the irreducible components of $\Bir(\mathbb{P}^2)_d$ by Bisi, Calabri, and Mella in 2015 \cite{BCM15}. In 2017, Blanc raised the question of what kind of algebraic structure can be put on $\Bir_k(X)$ in the survey \cite{B17}.

Let $R=k[x_0,\ldots,x_n]/\mathcal{I}(X)$ be the coordinate ring of $X$, and let $d$ be an integer. An element of $\Bir_k(X)$ can be presented by a set of $n+1$ homogeneous polynomials in $R$ of the same degree $d$. The {\it parameter space} is the space of coefficients of $(n+1)$ polynomials of degree $d$. Using the openness of loci introduced above,  we prove in  \autoref{Tmain} that for any irreducible projective variety $X$ over any infinite field $k$ and any integer $d$, $\Bir(X)_d$ has the structure of a constructible algebraic set. This result recovers \cite[Lemma 2.4 (2)]{BF13}. In \cite[Theorem 34]{pan2012some}, Pan and Rittatore mentioned the constructibility of some families of subsets of $\Bir(X)$ for a rational variety $X$. We show further in \autoref{Tmain}(iii) that $\edim(\Bir(X)_d)\leq (n+1)\HF_{R}(d)-1$ where $\HF_{R}$ is the Hilbert function of $R$.


The first subtlety in proving this result is defining the degree of a birational map over an arbitrary variety $X$. We introduced the concept of ``clear polynomial degree" of a rational map, as defined in  \autoref{Ddegree}. We want to generalize the notion of degree of birational maps for any variety, not just those with factorial coordinate rings. To do this, we need a natural and non-trivial way to measure the degree of a birational map. We took into account the topology introduced by Serre and the construction presented in \cite{BF13}.  In this way, a crucial step is to find a bound for the degree of the inverse of a birational map. This is the content of \autoref{Pinversedegree}.

  \autoref{C1} states that when $\HF_{\mathcal{I}(X)}(d)=0$, $\Bir(X)_d$ is a quasi-projective variety.  For instance, $\Bir(\mathbb{P}^n)_d$ is a quasi-projective variety for all $d$.

{\bf Acknowledgement}. The authors would like to thank C. Araujo, I. Pan, and J. Blanc for their helpful discussions and motivating questions on the topic.
Both authors are partially supported by the Coordena\c{c}\~{a}o de Aperfei\c{c}oamento de Pessoal de N\'ivel Superior-Brasil (CAPES)-Finance Code 001. The first author was partially supported by the MathAmSud project ``ALGEO'' and  CNPq  Grant Number 406377/2021-9.

\section{Open loci of ideals}
We will be using the following notation throughout this section: $k$ is a field, $S=k[x_0,\cdots,x_n]$ is a standard polynomial ring, and $R=S/\fp$ is a graded integral domain of dimension $r$. Let $d$ be an integer and let $M_1,\cdots,M_N$ be all of the monomials of degree $d$ in $S$, explicitly $N=\binomial{n+d}{d}$. For any $\aa=(a_1,\cdots,a_N)\in k^N$, we define $f_{\aa}= \sum_{i=1}^N a_iM_i\in S$.

Given a set of $m$ polynomials of degree $d$ in $S$, we can identify a corresponding point in the affine space $k^{mN}$. Specifically, let $g_1,\ldots,g_m$ be a sequence of homogeneous polynomials of degree $d$ in $S$. For any index $i$, there exists $\aa_i=(a_{i1},\cdots,a_{iN})\in k^N$ such that $g_i=f_{\aa_i}$. Conversely, any point $(\aa_1,\ldots,\aa_m)\in k^{mN}$ determines a sequence of polynomials $f_{\aa_1},\ldots,f_{\aa_m}\in S$ and, consequently, an ideal in $R$.

A locus of ideals is defined as the set of points $(\aa_1,\ldots,\aa_m)\in \mathbb{P}^{mN-1}$ such that the corresponding ideal $(f_{\aa_1},\ldots,f_{\aa_m})\subseteq R$ satisfies a particular property.

To begin, we state a computational algebraic property.

 \begin{Proposition}\label{Grobner} Let $m$ be an integer, $p(z_{1},\ldots,z_{m})\in k[z_{1},\ldots,z_{m}]$ a homogeneous polynomial, $\fb$  a homogeneous ideal of $S$ and keep the other notation as above. Then  $$V_{\fb}^{p}:=\{(\aa_1,\ldots,\aa_m)\in \mathbb{P}^{mN-1}_k:p(f_{\aa_1},\ldots,f_{\aa_m})\in \fb\}$$
 is a closed  algebraic  subset of $\mathbb{P}^{mN-1}_k$. Consequently, $U_{\fb}^{p}:=\{(\aa_1,\ldots,\aa_m)\in \mathbb{P}^{mN-1}_k:p(f_{\aa_1},\ldots,f_{\aa_m})\not\in \fb\}$ is open in $\mathbb{P}^{mN-1}_k$.
 \end{Proposition}
 \begin{proof} Fix a monomial order on $x_0,\ldots,x_n$ and consider a Groebner basis for the ideal $\fb$. Let $g=p(f_{\aa_1},\ldots,f_{\aa_m})$. If the remainder,  after doing the division algorithm, is zero then $g\in \fb$.  If the leading term of the polynomial $g$  is not divisible by those of the Groebner basis of $\fb$ then $g\not \in \fb$.  Hence, if the leading coefficient of $g$ is not divisible, we set the corresponding coefficient equal to zero and continue with the rest of the polynomial. Notice that this coefficient is a polynomial in terms of $a_i$'s. We continue dividing by  polynomials in the Groebner basis of  $\fb$. The reminder  is a  polynomial whose coefficients are polynomials in terms of $a_i$'s.  The result follows by induction on the number of terms of $g$. 
 \end{proof}
 \begin{Remark}\label{linearV} The special case where $m=1$ and $p(z)=z$ is of particular interest. In this case the elements of $V_{\fb}^{z}$ define  the vectors of the  $k$-vector space $\fb_{d}\setminus\{0\}$.  Hence $V_{\fb}^{z}$ is a linear subspace of $ \mathbb{P}^{N-1}_k$ of (Krull) dimension HF$_{\fb}(d)-1$ where HF$_{\fb}$ is the Hilbert function of $\fb$.

 \end{Remark}


Next, we study ideals that are traditionally called {\it ideals of the principal class}. That means an ideal whose number of generators is the same as its codimension.
\begin{Theorem} \label{Elimination} For any $1\leq j\leq \dim(R)$, the following set, $\mathcal{C}_j$, is a non-empty proper open subset of  $\mathbb{P}^{jN-1}$:
$$\mathcal{C}_j=\{(\aa_1,\ldots,\aa_j)\in \mathbb{P}^{jN-1}: \codim_R(f_{\aa_1},\ldots,f_{\aa_j})=j\}.$$ 

\end{Theorem}

 \begin{proof} 
 
 For any integer $t$, $S_t$ is the $k$-vector space generated by monomials of degree $t$.  We denote by $R_t$ the $k$-vector space generated by the class of the same monomials of $S_t$ modulo $\fp$ (these classes are not necessarily linearly independent).
 
   Recall that $R=S/\fp$ is a graded integral domain of dimension $r$. Let $\fp$ be generated by $p_1,\ldots,p_q$ of degrees $d_1,\ldots,d_q$ respectively.
    We first prove that $\mathcal{C}_r$ is open.  For any integer $m\geq \max\{d,d_1,\ldots,d_q\}$ and  $(\aa_1,\ldots,\aa_r)\in \mathbb{P}^{rN-1}$,  we define a map 
 \begin{align*}
\tau^m_{(\aa_1,\ldots,\aa_r)}:\underbrace{R_{m-d}\oplus \cdots\oplus  R_{m-d}}_{r}&\longrightarrow R_m\\
(g_1,\ldots ,g_r)&\mapsto (\sum_{j=1}^rg_jf_{\aa_j}).
\end{align*}
For simplicity, we use the notation $\tau^m$ instead of $\tau^m_{(\aa_1,\ldots,\aa_r)}$.

 $\tau^m$ is surjective if and only if $(x_0,\ldots,x_n)^m\subseteq (f_{\aa_1},\ldots,f_{\aa_r})+\fp$. The latter is equivalent to  $\codim_R(f_{\aa_1},\ldots,f_{\aa_r})=r$. 
 We show that the surjectivity of $\tau^m$ is an open property.

 Using the class of the basis of $S_i$'s to generate $R_i$'s, the presentation matrix of $\tau^m$ is a matrix with ${m+n\choose n}$ rows and ${m-d+n\choose n}r$ columns. One can regard the columns in $r$ partitions. The $i$th partition is described by considering the multiplication of the basis of  $R_{m-d}$ into $f_{\aa_i}$. Hence the entries of the columns of the  $i$'th partition are either $a_{ij}$  for some $j$, or zero. 
 
  Let assume $\dim_k(R_m)=N_m$ and let $I_{N_m}(\tau^m)$ be the ideal generated by the $N_m$-minors of  the representation matrix of $\tau^m$.  Then  $\tau^m$ is surjective if and only if $I_{N_m}(\tau^m)\neq 0$.  Consider the polynomial ring $A=k[a_{ij}:1\leq i\leq r,~~ 1\leq j\leq N]$ and $I_{N_m}(\tau^m)$ as an ideal of $A$.  Besides the standard graded structure of $A$, we consider an $\mathbb{N}^r$-graded structure for $A$ as follows: for all $1\leq j\leq N$ and $1\leq i\leq r$,  $$\deg(a_{ij})=(\underbrace{0,\ldots,\underbrace{1}_{i},0\ldots,0}_r).$$ 
  That means the columns of the  $i$'th partition are all of the same multi-degree.
  
  With respect to either grading, $I_{N_m}(\tau^m)$  is a ($\mathbb{N}^r$-)homogeneous ideal of $A$.

  Let $$\mathcal{I}=\sum_{m\geq\max\{d,d_1,\ldots,d_r\}}I_{N_m}(\tau^m).$$
Since $A$ is Noetherian, $\mathcal{I}$ is finitely generated. Let  $\mathcal{I}=(Q_1,\ldots,Q_v)$ for some ($\mathbb{N}^r$-)homogeneous polynomials $Q_i\in A$. We then have $\mathcal{C}_r=(V(Q_1,\ldots,Q_v))^c$ that is an open Zariski set of  $\mathbb{P}^{rN-1}$.

Now, fix  $j<r$.   Consider the polynomial ring $A'=k[a_{st}:1\leq s\leq j,~~ 1\leq t\leq N]$, let $\aa_i$ to be the sequence  of variables $a_{i,1},\ldots,a_{i,N}$ for all $i$, and define
 $$\mathcal{P}_j:=(Q_i(\aa_1,\ldots,\aa_j,\A_{j+1},\ldots,\A_r)\in A': 1\leq i\leq v, \A_t\in \mathbb{P}^{N-1} \text{ and }\{ f_{\A_{j+1}},\ldots,f_{\A_r}\}\in \mathfrak{S})$$
 where $\mathfrak{S}=\{\{f_{\A_{j+1}},\ldots,f_{\A_r}\}: f_{\A_{j+1}},\ldots,f_{\A_r} \text{~~is a part of a system of parameters of~~} R\}$.
We claim  that $\mathcal{C}_j=(V(\mathcal{P}_j))^c\subseteq \mathbb{P}^{jN-1}$.

We first notice that $\mathcal{P}_j$ is a homogeneous ideal of $A'$. This is due to the $\mathbb{N}^r$-graded structure of $A$ and homogeneity of  $I_{N_m}(\tau^m)$,   as explained in the previous paragraph. So that $(V(\mathcal{P}_j))^c\subseteq \mathbb{P}^{jN-1}$.

Since $R$ is a quotient of a polynomial ring, $(\A_1,\ldots,\A_j)\in \mathcal{C}_j$ if and only if   $(f_{\A_1},\ldots,f_{\A_j})$ is a part of a system of parameters of $R$ that is $\dim(R/(f_{\A_1},\ldots,f_{\A_j}))=r-j$. Hence there exists $(\A_{j+1},\ldots,\A_r)\in \mathfrak{S}$  such that  $\codim(f_{\A_1},\ldots,f_{\A_j},\ldots,f_{\A_r})=r$; so that there exists $1\leq i\leq v$  with $Q_i({\A_1},\ldots, {\A_j},\ldots, {\A_r})\neq 0$. The latter says that $\mathcal{C}_j\subseteq (V(\mathcal{P}_j))^c$.  The other inclusion is more visible. 
\end{proof}

 Now, we pass from "codim" to "grade". 
 \begin{Proposition}\label{LG2} Keeping the same notation in this section, the set 
 $$G_{2}:=\{(\aa,\bb)\in \mathbb{P}^{2N-1}_k: (f_{\aa},f_{\bb}) \text {~~is a regular sequence in~~} R\} $$ is a Zariski    open subset of $ \mathbb{P}^{2(N-1)}_k$.  $G_2$ is non-empty if and only if $\depth(R)\geq 2$.

  \end{Proposition}
 \begin{proof}
 It is known that for a ring such as $R$ (homomorphic image of Gorenstein ring), the $(S_2)$-locus of $\Spec(R)$ is an open subset, c.f. \cite[12.3.10]{BS}.  Let $V(I)$ be the non-$(S_2)$-locus of $R$ then   $I$ is a homogeneous ideal of $R$ of codimension at least $2$.
 
   Let  $\mathcal{A}=\{\fq\in \Spec(R):\codim(\fq)=2 \text{ and} \grade(\fq)=1\}$. $\mathcal{A}$ may be empty.  If not,   for any $\fq\in \mathcal{A}$, $R_{\fq}$ is not $S_2$, hence $\fq\supset I$, and since $\codim(I)\geq 2$, $\fq$ is a minimal prime over $I$.  
   Thence $\mathcal{A}$ is a finite set (or empty). 
 
 Recalling the definition of $\mathcal{C}_2$ from  \autoref{Elimination}, 
 $$G_2=\mathcal{C}_2\setminus\{(\aa,\bb)\in \mathbb{P}^{2N-1}_k: \codim(f_{\aa},f_{\bb})=2 \text{~and~} \grade(f_{\aa},f_{\bb})=1\}.$$
 
 We show that the take-out part is closed. 
 \begin{align*}
 \{(\aa,\bb)\in \mathbb{P}^{2N-1}_k: \codim(f_{\aa},f_{\bb})=2 \text{~and~} \grade(f_{\aa},f_{\bb})=1\}&=\\
 \{(\aa,\bb)\in \mathbb{P}^{2N-1}_k: \codim(f_{\aa},f_{\bb})=2 ~\text{and} ~ f_{\aa},f_{\bb}\in \fq~ \text{~for some~} \fq\in \mathcal{A}\}.
 \end{align*}
 By  \autoref{Grobner}, $f_{\aa},f_{\bb}\in \fq$ if and only if $(\aa,\bb)\in V^{z}_{\fp+\fq}\times V^{z}_{\fp+\fq}$ where $\fp$ is the defining ideal of $R$. Therefore 
 $$G_2=\mathcal{C}_2\setminus \bigcup_{\fq\in \mathcal{A}}(V^{z}_{\fp+\fq}\times V^{z}_{\fp+\fq}).$$
This completes the proof since $\mathcal{A}$ is a finite set and $\mathcal{C}_2$ is open by  \autoref{Elimination}.

  \end{proof}
 Initially, it might appear enticing to establish both \autoref{Elimination} and \autoref{LG2} through  ``prime avoidance lemma". However, it's worth noting that the prime avoidance lemma constructs the sought-after sequences in an inductive manner. On the contrary, the above mentioned results construct these desired sequences simultaneously.

Should we opt for an inductive approach, it would necessitate a considerable endeavor to demonstrate the openness of the assembled sets.
 
The next locus we discuss is about {\it analytic spread}. We present more details about its definition.
Let $(A,\fm)$ denote a standard graded Noetherian algebra over the field $k$ and   $I=(f_1,\ldots,f_s)$ stand for a graded ideal of $A$. Set $\mathcal{R}_A(I):=A[It]\subset A[t]$, the Rees algebra of $I$ over $A$.  The special fiber of $I$ is the algebra $\mathcal{F}(I)=\mathcal{R}_A(I)/\fm \mathcal{R}_A(I)$. The Krull dimension of $\mathcal{F}(I)$ is called the {\em analytic spread} of $I$. When $I$ is generated in a single degree, $\mathcal{F}(I)\simeq k[f_1,\ldots,f_s]$.   We say that  $I$ has maximum  analytic spread  if $\ell_A(I)=\min\{\Dim(A),s\}$. For more details on this notation see \cite{HunekeSwanson}. The maximum analytic spread is closely related to the  Noether normalization theorem. 
 \begin{Theorem}\label{LNn+1}
  Let $k$ be an infinite field. The following set that   we call  the {\it  Noether Normalization locus}  is  a non-empty open subset of  $\mathbb{P}^{(n+1)N-1}$
  $$\mathcal{N}_{n+1}=\{(\aa_0,\ldots,\aa_n)\in \mathbb{P}^{(n+1)N-1}:\ell_R(f_{\aa_0},\cdots,f_{\aa_n})=r\}.$$
   \end{Theorem}
   
\begin{proof}   We prove by induction on $i\geq r$ that $\mathcal{N}_i$ is a Zariski open where
$$\mathcal{N}_i=\{(\aa_0,\ldots,\aa_{i-1})\in \mathbb{P}^{iN-1}:\ell_R(f_{\aa_0},\cdots,f_{\aa_{i-1}})=r\}.$$  

First,  we verify  the case $i=r$. Let $\{\theta_1,\ldots,\theta_r\}\subset R_d$ such that $\Dim(k[\theta_1,\ldots,\theta_r])=r$. Since $k[\theta_1,\ldots,\theta_r]$ is integrally closed and the extension $k[\theta_1,\ldots,\theta_r]\subset R$ is integral, the going-down theorem holds for this extension, in particular, the ideal $(\theta_1,\ldots,\theta_r)R$ has codimension $r$ i.e it is a system of parameters. On the other hand, any system of parameters is analytically independent \cite[Theorem 14.5]{M86}. Therefore 
$$\mathcal{N}_r=\{(\aa_0,\ldots,\aa_{r-1})\in k^{rN} : \codim_R(f_{\aa_0},\cdots,f_{\aa_{r-1}})=r\}=\mathcal{C}_r.$$
The latter determines an open subset of $\mathbb{P}^{rN-1}$ by  \autoref{Elimination}. 

Next, we show that $\mathcal{N}_{r+1}$ is open Zariski. The proof of the general inductive pass is similar.

Let $\aa_1,\ldots,\aa_{r+1}\in k^{N}$ be such that $\ell_R(f_{\aa_1},\ldots,f_{\aa_{r+1}})=r$. Then using the same technique as that of the proof of the  Noether normalization theorem \cite[Lemma 2, page 262]{M86}, one can show that there exists $(c_1,\ldots,c_{r})\in k^r$ such that $\ell_R(f_{\aa_1}-c_1f_{\aa_{r+1}},\ldots,f_{\aa_r}-c_rf_{\aa_{r+1}})=r$--(here we need an infinite field).  We repeat this statement in other words: we denote the points in $k^{(r+1)N},$ $k^{rN}$ and $k^{r(r+1)}$ by matrices of sizes $(r+1)\times N$, $r\times N$ and $r\times (r+1)$, moreover,  $M=(m_1,\ldots,m_N)^t$ the $N\times 1$ matrix of the monomials of degree $d$ in $R$. Then for any $(r+1)\times N$-matrix $A=(\aa_{ij})$ with $\ell_R(AM)=r$, there exists an $r\times {(r+1)}$-matrix $C$ such that $CA\in \mathcal{N}_r$. 

Conversely, if for an $(r+1)\times N$-matrix $A=(\aa_{ij})$  there exists an $r\times {(r+1)}$-matrix $C$ such that $CA\in \mathcal{N}_r$. Then $r=\ell(CAM)\leq \ell(AM)\leq \ell(M)=r$. So that the rows of $A$ constitute  points of $k^{(r+1)N}$ that belong to $\mathcal{N}_{r+1}$.
In brief 
\begin{equation}\label{nr+1}
\mathcal{N}_{r+1}=\{A\in  k^{(r+1)N}: \text{~~there exists~~}  C\in k^{r(r+1)} \text{~~with~~} CA\in \mathcal{N}_r\}.
\end{equation}
Restating the above equality, any $r\times(r+1)$ matrix $C$ determines a linear map $\varPhi_C:k^{r+1}\to k^r$.  $N$ copies such a map define a linear map
 $$\bigoplus^N\varPhi_C:k^{(r+1)N}\to k^{rN}.$$ 
The equality (\autoref{nr+1}) states that
$$\mathcal{N}_{r+1}=\bigcup_C (\bigoplus^N\varPhi_C)^{-1}(\mathcal{N}_r).$$ 
The linearity of $\bigoplus^N\varPhi_C$ guarantees that $\mathcal{N}_{r+1}$ is an open Zariski subset of $\mathbb{P}^{(r+1)N-1}$.
The inductive pass of the induction argument  follows by showing that 
$$\mathcal{N}_{r+j}=\bigcup_C (\bigoplus^N\varPhi_C)^{-1}(\mathcal{N}_{r+j-1}).$$ 
 \end{proof}
 
In the following remark, we provide an alternative, constructive proof for   \autoref{LNn+1}.
\begin{Remark}
Consider three sets of independent variables
$$ \begin{pmatrix} 
  Z_{1,1}&\ldots&\ldots&Z_{1,N}\\
  \cdots&\cdots&\cdots&\cdots\\
  Z_{r,1}&\ldots&\ldots&Z_{r,N}
 \end{pmatrix},
 \quad
 \begin{pmatrix} 
  X_{1,1}&\ldots&\ldots&X_{1,N}\\
  \cdots&\cdots&\cdots&\cdots\\
  X_{r,1}&\ldots&\ldots&X_{r,N}\\
  X_{r+1,1}&\ldots&\ldots&X_{r+1,N}
 \end{pmatrix},
 \quad
 \begin{pmatrix} 
  Y_{1,1}&\ldots&Y_{1,r+1}\\
  \cdots&\cdots&\cdots\\
  Y_{r,1}&\ldots&Y_{r,r+1}
 \end{pmatrix},
 $$
  and  set $\ZZ_i=(Z_{i,1},\ldots, Z_{i,N})$, $\XX^i=(X_{i,1},\ldots,X_{r+1,i})^t$ and $\YY_i=(Y_{i,1},\ldots,\YY_{i,r+1})$.
   Assume that $\mathbb{P}^{rN-1}\setminus \mathcal{N}_r=V(\fn_r)$ where $\fn_r\subset k[\ZZ_1,\ldots,\ZZ_r]$ is a homogeneous ideal.
   For any $F\in \fn_r$, set 
   $$g_F(\YY_1,\ldots,\YY_r,\XX^1,\ldots,\XX^N)=F((\YY_1\XX^1,\ldots,\YY_1\XX^N),\ldots,(\YY_r\XX^1,\ldots,\YY_r\XX^N))).$$
 Let $\fn=(g_F: F \in  \fn_r)$ be the homogeneous ideal in  $k[\YY_1,\ldots,\YY_r,\XX^1,\ldots,\XX^N]$.
 The projection $\pi:\Proj(k[\YY_1,\ldots,\YY_r,\XX^1,\ldots,\XX^N])\to \Proj(k[\XX^1,\ldots,\XX^N])$ is a closed map. Hence, there is a homogeneous ideal $\fn_{r+1}\subset k[\XX^1,\ldots,\XX^N]$ with $V(\fn_{r+1})=\pi(V(\fn))$. 
 
 The explanation of $\mathcal{N}_{r+1}$ in (\autoref{nr+1}) shows that  $\mathbb{P}^{(r+1)N-1}\setminus \mathcal{N}_{r+1}=V(\fn_{r+1}),$ as desired .

\end{Remark}
 
 \section{Application to Birational Maps}
  We fix a field $k$, of any characteristic and not necessarily algebraically closed. Let $X\subseteq \mathbb{P}^n_k$  be a reduced and irreducible projective variety.   We assume $X$ is  non-degenerated i.e. $X$ is not contained in any hyperplane, and denote  the coordinate ring of $X$ by $R$. I.e. $R:=k[x_0,\ldots,x_n]/\mathcal{I}(X)$.
 
 \subsection{Clear Degree}
 
 A rational map  $h$ from $X$ to $X$ is presented  by $$h:(x_0:\cdots:x_n)\mapsto (h_0(x_0:\cdots:x_n):\cdots:h_n(x_0:\cdots:x_n)).$$
 Here $h_0,\cdots, h_n$ belong to $R$ and  all of  them are of the  same degree. One faces preliminarily the fact that the map can be defined by different sets of forms of fixed degree, where this degree may vary from set to set. 
  One way to define the degree of a birational map is to consider the minimum possible degree among all of the representative of the map. This is the natural way, that coincide with the usual way of definition for the case where $X=\mathbb{P}^n$. However, in the case of arbitrary variety $X$, such as $X_2$  in the  example below,  it is possible that a map admits two or more representative of minimal degree. Non-uniqueness of the representative  causes difficulties to study the locus of birational maps.

   We present two examples:
 \begin{Example} \label{Ex1}  Consider  rational maps $
\sigma_i:X_i\dasharrow X_i$, $(x:y:z)\mapsto (yz:xz:xy),$ 
in the  cases where  $X_1=V(y^2-xz)$, a smooth curve in $\mathbb{P}^2$, and 
$X_2=V(y^3-x^2z)$ a cuspidal curve  in $\mathbb{P}^2$.
 $\sigma_1$  is also represented by $(z:y:x)$; while $\sigma_2$ admits another representative of degree $2$ that is not obtained by multiplication, namely $\sigma_2$ is the same as the map 
 \begin{align*}
\tau:X_2&\dasharrow X_2\\
(x:y:z)&\mapsto (xz:y^2:x^2).
\end{align*} 
 \end{Example}

 A strategy to circumvent the complexities associated with defining  degree of a birational map is to consider the grade of the ideal of the base of the map. Recall that for an ideal $I$ in a ring $R$, $\grade(I,R)$ is the maximum length of a proper regular sequence inside $I$. For the map $h$ defined earlier, let $I=(h_0,\ldots,h_n)\subset R$. If $\grade(I,R)\geq 2$, then the map $h$ is presented uniquely up to multiplication by an element of $R$ (rather than an element of  $Q(R)$), as shown in \cite[Proposition 1.1]{Si04}. Based on this, we present the following definition:
 
 \begin{Definition}\label{Ddegree} Let $h:X\dasharrow X$ be a rational map.  We say the map $h$ has {\it clear polynomial  degree} $d$, if $h$ admits a representative  $(h_0:\cdots:h_n)$ such that  the ideal in $R$ generated by $h_0,\ldots,h_n$ has grade at least $2$ and  $\deg(h_i)=d$.
  \end{Definition}
  One says  $X$ satisfies the Serre's $(S_2)$-condition if  for all prime ideal $\fp\subset R$, 
$\depth(R_{\fp})\geq \min\{\codim(\fp),2\}.$ In the case where $X$ is $S_2$ or at least $S_2$ in codimension $2$, then   $\grade(I,R)\geq 2$ if and only if $\codim(I)\geq 2$.  In particular, if $R$ is factorial, then $\grade(I,R)\geq 2$ if and only if $h_0,\ldots,h_n$ have no common factor.

 

   \begin{Definition}\label{DBird} We denote by $\Bir(X)_d$ the set of birational maps from $X$ to $X$  of clear polynomial degree $d$.  $\Bir(X)_{\ast}:=\cup_{d=1}^{\infty}\Bir(X)_d$.   If, for example, $R$ is factorial, then $\Bir(X)_{\ast}=\Bir(X)$. 
\end{Definition}
In  \autoref{Ex1}, $\sigma_1\in \Bir(X_1)_1\setminus \Bir(X_1)_2 $ while $\sigma_2\in \Bir(X_2)\setminus \Bir(X_2)_*$.

\begin{Remark}  Once we study the entire group $\Bir(X)$, it is possible and useful to replace $X$ with a smooth variety, using the resolution of singularities, or in some cases even replace $\Bir(X)$ with $\mathrm{Aut}(X)$, using the minimal model. However, such a replacement is not as functional in the study of $\Bir(X)_{d}$. For example, the varieties $X_1$ and $X_2$ in  \autoref{Ex1} are birationally equivalent. The maps $\theta:X_2\dasharrow X_1$ that sends $(x:y:z)$ to $(yz:xz:y^2)$ and $\theta^{-1}:X_1\dasharrow X_2$ that sends $(x:y:z)$ to $(yz:xz:x^2)$ define the birational equivalence. Using these birational maps, we have an isomorphism $\Psi:\Bir(X_1)\to \Bir(X_2)$ given by $\Psi(\phi)=\theta^{-1}\circ\phi\circ\theta\in \Bir(X_2)$. Interestingly, in this example, $\Psi(\sigma_1)=\sigma_2$. Thus the restriction of $\Psi$ to $\Bir(X_1)_*$ does not define a well-defined map to $\Bir(X_2)_*$.
\end{Remark}
A natural question is the following
\begin{Question}
For which variety $X$ does $\Bir(X)_*$ admit the structure of a group?
\end{Question}
An immediate answer to this question is the varieties $X$ whose coordinate ring $R$ is a unique factorization domain, since $\Bir(X)_*=\Bir(X)$. 

A remarkable class of varieties to be studied is the class of projective smooth complete intersection $F$:    when $\Dim(F)\geq 3$ the coordinate ring of $F$ is a factorial ring due to the theorem of Grothendieck c.f. \cite{Call94};  when the  sum of  degrees is at most $n$ these varieties are Fano varieties. Recall that for a Fano variety $F$, $\Bir(F)$ is not  a finite set, unlike varieties of general type. 

When $X$ is a curve,  $\Bir(X)_*$ coincides with linear automorphism over $X$. Hence it is a group. We prove the latter in the following proposition.
\begin{Proposition}\label{X curve} Let $X$ be an irreducible curve. The following statements are equivalent
\begin{itemize}
\item{$R=k[x_{0},\ldots,x_{n}]/\mathcal{I}(X)$ is a Cohen-Macaulay ring (of Krull dimension $2$).}
\item{$\Bir(X)_*\neq \emptyset$.}
\item{$\Bir(X)_*$ is the same as linear automorphisms of $X$.  }
\end{itemize}
\end{Proposition}
\begin{proof} 

 The identity map has clear polynomial degree if and only if the irrelevant maximal ideal has  grade at least $2$. Since $R$ is of dimension $2$,  it is equivalent to Cohen-Maculayness of $R$.  Moreover, the base ideal of any rational map is contained in the irrelevant maximal ideal, hence if $\Bir(X)_*\neq \emptyset$ then the identity map has clear polynomial degree. This proves the equality of the first two parts. The third statement clearly implies the second one. Hence it remains to show that the first statement implies the third one. 
  
  Now assume that $R$ is Cohen-Macaulay. To prove the last assertion, we use the birationality criterion of Simis-Ulrich-Vasconcelos \cite[Theorem 6.6]{SUV01}.  Since the latter criterion holds in more general cases than our setting,  we present the adapted version for our purpose in  \autoref{SUV}.  Let $\phi:X\dasharrow X$ be a dominant rational map of clear degree $d$. Regarding the notation of  \autoref{SUV}, $r=2$, $g=\Ht(I)= 2$ and $\mathfrak{R}=\emptyset$.Therefore, $e(R)= e(R)d$  if and only if $\phi$ is birational.  That is $d=1$ if and only if $\phi$ is birational.

\end{proof}

\begin{Theorem}\label{SUV}\cite[Theorem 6.6]{SUV01}. Let $X$ be an irreducible projective variety with coordinate ring $R$ of dimension $r$. Let $\phi:X\dasharrow X$ be a dominant rational map of polynomial degree $d$.
Let $I$ be the ideal generated by that presentation. Set  $g=\Ht(I)$ and  $\mathfrak{R}:=\{\fp\in \Proj(R): \fp\supseteq I {\text ~and~} \dim(R/P)=\dim(R/I)\}.$  Then 
$$e(R)\leq e(R)d^{r-1}  -  \sum_{\fp\in \mathfrak{R}}e(I_{\fp},R_{\fp})e(R/P)d^{r-g-1}.$$
Equality happens if and only if $\phi$ is birational and $r\leq g+1$.

\end{Theorem}

In dimension $2$, there are some varieties that have nice properties but do not have a group structure on their birational maps. We used Macaulay-2\cite{grayson1998macaulay} to construct the following example.

\begin{Example}\label{veronese} Let $V\subset \mathbb{P}^5$ be the Veronese surface realized by the Veronese embedding. Let $V_1$ be the projection of $V$ into $\mathbb{P}^4$. $V_1$ is the image of the following map $\nu$ 
$$
 \begin{array}{rcl}
  \nu:\mathbb{P}^2&\dasharrow& V_1\\
(s:t:u)&\mapsto& (st:su:t^2:tu:u^2)
\end{array}
$$ 
$\nu^{-1}:V_1\dasharrow \mathbb{P}^2$ is given by $(a_0:\cdots:a_4)\mapsto (a_1:a_3:a_4)$.
Consider the birational map 
$f:\mathbb{P}^2 \dasharrow\mathbb{P}^2$, $
(s:t:u)\mapsto (st:su:t^2).$ Let $\varphi=\nu f\nu^{-1}:V_1\dasharrow V_1$. $\varphi$ is a birational map represented by $$(a_0a_1:a_0a_2:a_1^2:a_0a_3:a_2^2).$$
Calculations show that $\varphi$ has clear degree $2$. However, having a clear degree does not extend either to its inverse or its successive powers. This example shows that $\Bir(V_1)_*$ does not admit any structure as a group or a monoid. Furthermore, we mention that $V_1$ is a quadric rational smooth surface. Its coordinate ring is Cohen-Macaulay and not factorial.
 
\end{Example} 

In our experience with Veronese surface, we could not find any such map, as the one in   \autoref{veronese}. However, we still could not show that $\Bir(V)_*$ is a group.

\subsection{The structure of $\Bir(X)_d$}
The following proposition has a critical role in the proof of  \autoref{Tmain}. It has been already stated in \cite[Proposition 1.2]{HS17}, and its proof is based on the main result of  \cite{DHS12}. 
\begin{Proposition}\label{Pinversedegree}   Let $X\subseteq \mathbb{P}_k^n$ denote a closed reduced and irreducible subvariety, and  $\mathfrak{F}:X\dasharrow  \mathbb{P}_k^m$ a birational map  admitting a representative of polynomial degree $d\geq 1$. Let $Y=\Image \mathcal{F}$. Assume that $X$ and $Y$ are non-degenerated and $\delta=\max\{d+1,d_0\}$ where  $d_{0}$ is the maximum degree among the degrees of a minimal generating set of  $\mathcal{I}(X)$.
 
 Then the inverse map  $\mathfrak{F}^{-1}$ admits a representative of  polynomial degree at most
 	\begin{equation}\label{bound}
	 B(X,d,n)=2n\bigg[\frac{1}{2}\delta^{2(n+m+1-\Dim(X))^2}+\delta\bigg]^{2^{(\Dim(X)+2)}}.
	 \end{equation}
 	
\end{Proposition}
In the case where $X=Y$, accordingly $n=m$,  we call the formula in  \autoref{bound},  the inverse bound and denote it by $B(X,d)$.

In terms of determining an upper bound,  this  result is a generalization of a result attributed to O.Gabber. Gabber's result  states that for  $\mathfrak{F}: \mathbb{P}^n\dasharrow  \mathbb{P}^n$ of degree $d$ the inverse map is of degree at most $d^{n-1}$. It is clear that the bound $B(\mathbb{P}^n,d)$, above,  is much higher than $d^{n-1}$.

 To prove the main theorem,  we adapt some of the notation of   \cite[Lemma 2.4]{BF13}.   Returning the notation in Section 2, we  fix an integer $d$ and denote by $M_1,\ldots,M_N$  all of the monomials of degree $d$ in the polynomial ring $S=k[x_0,\cdots,x_n]$.
 Consider the projective space $\mathbb{P}_k^{(n+1)N-1}$. Any point $\aa$ of $\mathbb{P}_k^{(n+1)N-1}$ is an  $(n+1)$-uple  $(\aa_0,\ldots,\aa_n)$ with $\aa_i\in \mathbb{A}^N_k$. For each index $\aa_i$ we construct a polynomial 
 \begin{equation}\label{fa}
 f_{\aa_i}= \sum_{j=1}^N (\aa_i)_jM_j.
 \end{equation}
 We denote $W_d$ to be the set of equivalence classes of nonzero $(n+1)$-uple $(f_0,\ldots,f_n)$ of homogeneous polynomials $f_i \in S$ of degree $d$ where $(f_0,\ldots,f_n)$ is equivalent to $(\lambda f_0,\ldots,\lambda f_n)$ for any $\lambda\in k\setminus{\{0\}}$. The equivalence class of $(f_0,\ldots,f_n)$ will be denoted by $(f_0:\ldots :f_n)$. It is clear that  $\mathbb{P}^{(n+1)N-1}$   is isomorphic to $W_d$ via the correspondence:
 \begin{align*}
 \theta:\mathbb{P}^{(n+1)N-1}&\to W_d\\
 \aa=(\aa_0,\ldots,\aa_n)&\mapsto f_{\aa}=(f_{\aa_0}:\cdots:f_{\aa_n}).
 \end{align*}

We define $H_d\subseteq W_d$ to be the set of elements $f=(f_0,\ldots,f_n)\in W_d$ such that the rational map 
\begin{align}
\psi_f:X&\dasharrow X\\
(x_0:\cdots:x_n)&\mapsto  (f_0(x_0,...,x_n) : \cdots: f_n(x_0,...,x_n))\nonumber
\end{align} 

  is a well-defined birational  map of {\it polynomial degree }$d$. 
  
  $\tilde{H}_d\subseteq H_d$ is the set of elements  $f=(f_0,\ldots,f_n)\in H_d$ such that $\psi_f$ is birational of {\it clear polynomial degree} $d$.  
  
  The last concept before stating the main theorem is the definition of {\it constructible set}. A subset of a topological space is called {\it constructible} when it is a finite union of {\it locally closed} sets. A set is {\it locally closed}  if it is an intersection of a closed set  and an open set. For a constructible subset $\mathcal{C}$ of a projective space. The embedding dimension of $\mathcal{C}$ is  defined by $ \edim(\mathcal{C}):=\min\{m:\mathcal{C}\subseteq \mathbb{P}^{m} \text{~~as a closed embedding} \}$.


\begin{Theorem}\label{Tmain}  Let $d$ be an integer,  $k$  an infinite field and $X\subseteq \mathbb{P}^n$ an irreducible projective non-degenerated variety with coordinate ring $R$ and Hilbert function $\HF_{R}$.   Then 
\begin{itemize}
\item[(i)]{$H_d$ and $\tilde{H}_d$ (defined above) are  quasi-projective varieties;}
\item[(ii)]{$\Bir(X)_d$ has the structure of  a constructible algebraic set; }
\item[(iii)]{$\edim(\Bir(X)_d)\leq (n+1)\HF_{R}(d)-1$.}
\end{itemize}
\end{Theorem}
\begin{proof} (i) The first paragraph of the proof follows the argument of \cite[Lemma 2.4]{BF13}. Let $R=k[x_0,\cdots,x_n]/\fp$ be the coordinate ring of $X$. Let  $\fp=(p_1,\ldots,p_t)$ and $r=\Dim(R)=\Dim(X)+1$. Let $d'=B(X,d)$ be the inverse bound introduced in  \autoref{Pinversedegree}. We denote by $Y'\subseteq W_{d'}\times W_d$ the set consisting of elements $(g, f )$ such that $h := (g_0(f_0,\ldots,f_n),\ldots, g_n(f_0,\ldots,f_n))$  and $h' := (f_0(g_0,\ldots,g_n),\ldots, f_n(g_0,\ldots,g_n))$ are  multiples of the identity modulo $\fp$; i.e., $h_ix_j-h_jx_i\in \fp$  and $h'_ix_j-h'_jx_i\in \fp$ for all $i,j$.

In order to $f:X\dasharrow X$ and $g:X\dasharrow X$ be well-defined maps, the coordinates must satisfy $p_i(f_0,\ldots,f_n)\in \fp$ and $p_i(g_0,\ldots,g_n)\in \fp$ for all $i$. 

Let $$\mathcal{F}_i:=\{(f_{\aa_0},\ldots,f_{\aa_n})\in W_d:p_i(f_{\aa_0},\ldots,f_{\aa_n})\in \fp\}$$ and $$\mathcal{G}_i:=\{(f_{\aa'_0},\ldots,f_{\aa'_{n}})\in W_{d'}:p_i(f_{\aa'_0},\ldots,f_{\aa'_n})\in \fp\}.$$
  \autoref{Grobner}  shows that   $\mathcal{F}_i $ and $\mathcal{G}_i$ are closed varieties of $W_{d}$ and $W_{d'}$ respectively.  Similarly, $Y'$ is a closed subset of $W_{d'}\times W_d$. Hence 
$$Y:=Y'\cap (\bigcap_{i=1}^t \mathcal{F}_i \times W_{d'})\cap (W_d\times \bigcap_{i=1}^t \mathcal{G}_i )$$
is a closed subset of $W_{d'}\times W_d$. Now, applying the principal theorem of the elimination theory,  the image of $Y$ by the projection $\pi_2:W_{d'}\times W_d\to W_d$ is closed in $W_d$.

Next, we look for $f=(f_0:\ldots:f_n)\in W_d$ such that the  rational map $f:X\dasharrow X$ is dominant. Being dominant  means that the Krull dimension of the image of $f$ is $r-1$. The coordinate ring of the closure of the image of $f$ is $\Proj(\mathcal{F}_R(I))$ (see definitions  after  \autoref{Pinversedegree}) wherein $I$ is the ideal generated by the classes of $f_0,\ldots,f_n$ in $R$. Therefore the map $f$ is dominant if and only if $\ell_R(I)=r$.  According to  \autoref{LNn+1}, this defines the set $\mathcal{N}_{n+1}$. Therefore, 
$$H_d=\{h\in W_d: \psi_h:X \dasharrow X \text{~~is birational of polynomial degree~~} d\}=\pi_2(Y)\cap \theta(\mathcal{N}_{n+1}).$$
That is an open subset of the algebraic set  $\pi_2(Y)$ thence a quasi-projective variety. 

Next,  we show that $\tilde{H}_d$ is a quasi-projective variety. $\tilde{H}_d$ is the intersection of  $H_d$ with those birational maps whose base ideal is of  grade at least $2$.  The regular sequence of length $2$ may not be among the generators, however it is inside a linear combination of them. Hence we need to work with
$$U_2=\{(\aa_0,\ldots,\aa_n)\in \mathbb{P}^{(n+1)N-1 }:\sigma(\aa_0,\ldots,\aa_n)\in G_2\times \mathbb{A}^{(n-1)N} \text{~for some~} \sigma\in {\rm GL}_{n+1}(k)\}.$$ 
That is 
$$U_2=\bigcup_{\sigma\in {\rm GL}_{n+1}(k) }\sigma^{-1}(G_2\times \mathbb{A}^{(n-1)N}).$$
Since  $G_2$ is an open  subset,  by  \autoref{LG2}, $U_2\subseteq \mathbb{P}^{(n+1)N-1}$ is open, as well.

\medskip
We now have $\tilde{H}_d=\pi_2(Y)\cap \theta(\mathcal{N}_{n+1}\cap U_2)$ thence $\tilde{H}_d$ is  a  quasi-projective variety.
\medskip

(ii)  In the case where $X=\mathbb{P}^n$, the proof of this part is to show that there is a bijection between $\tilde{H}_d$ and $ \Bir(X)_d$.  However, when $X\neq \mathbb{P}^n$, even if  a map $f=(f_0,\ldots,f_n):X\dasharrow X$ belongs to $\tilde{H}_d$, the defining polynomials of $f$ in $S$ are not uniquely defined. In other words, two sets of polynomials  in $S$ say $f_0,\ldots,f_n$ and $f'_0,\ldots,f'_n $,  define the same map over $X$ if $f_i-f'_i\in \fp$ for all $i$. 

To address this ambiguity, we revisit \autoref{linearV}.
Let 
\begin{equation}\label{s}
s=N-{\HF}_{\fp}(d).
\end{equation}
 Then the ideal of definition of the linear variety  $V_{\fp}^{z}:=\{(a_1,\ldots,a_N)\in \mathbb{P}^{N-1}_k:\sum_{i=1}^N a_iM_i\in \fp\}$  is determined by  $s$ linear equations $l_1,\ldots,l_s$.  
 Let $\mathcal{L}:k^N\to k^s$ be the linear map defined by $l_1,\ldots,l_s$. By repeating $\mathcal{L}$ on the diagonal, we extend it to a map $\mathcal{L}':k^{(n+1)N}\to k^{(n+1)s}$ thence to a map $\mathcal{D}: W_d\dasharrow \mathbb{P}^{(n+1)s-1}$. 
 
 More precisely, let $y_1,\ldots,y_N$ be the coordinates of $\mathbb{A}_k^N$, $l_i=\sum_{j=1}^N l_{ij}y_j$ for $i=1,\ldots,s$  and  $(f_{\aa_0}:\ldots :f_{\aa_n})\in W_d$. Then
  $$\mathcal{D}(f_{\aa_0}:\cdots :f_{\aa_n})=(l_1(\aa_0):\cdots:l_s(\aa_0):\cdots:l_1(\aa_n):\cdots:l_s(\aa_n)).$$
  Let $\mathcal{D}'=\mathcal{D}|_{\tilde{H}_d }$.

Notice that $\mathcal{D}'$ is a (regular) morphism. Since for any $0\leq j\leq n$, the linear forms $l_i(\aa_j)$  all vanish for $0\leq i\leq s$ if and only if $f_{\aa_j} \in \fp$.  This provides a contradiction to the choice of $(f_{\aa_0}: \cdots: f_{\aa_n})$ as a rational map over $X$ of grade at least $2$ in the coordinate ring $R=S/ \fp$.

Recalling the definition of the clear polynomial degree of a birational map, we have the natural  map 
\begin{align*}
\Psi_d:{\tilde H}_d&\longrightarrow \Bir(X)_d\\
h&\mapsto  \psi_h.
\end{align*} 
  The map $\Psi_d$ induces a well-defined map $\Phi$, according to the following triangle 
\begin{equation}\label{diagram}
\xymatrix{
	& \tilde{H}_d \ar^{\Psi_d}[rd]                         &                   \\
	\Image(\mathcal{D}')\ar@^{<-} [ru]^{\mathcal{D}'} \ar[rr]^{\Phi } &   & \Bir(X)_d.  \\    
}
\end{equation}

We verify that $\Phi$ is bijective.  

Well-definedness: Let $g=\mathcal{D}'(f)=\mathcal{D}'(f')$ for  $f=(f_{\aa_0}:\cdots:f_{\aa_n})$ and $f'=(f_{\aa'_0}: \ldots: f_{\aa'_n})$ in $\tilde{H}_d $.  We have to show that $\Phi(g):=\Psi_{d}(f)=\Psi_{d}(f')=:\Phi(g)$. 
Since $\mathcal{D}'$ is a linear map, $\mathcal{D}'(f-f')=0$. Therefore, $l_j(\aa_i-\aa_i')=0$ for all $i$ and $j$.  Hence $\aa_i-\aa_i'\in V_{\fp}$. Equivalently, $f_{\aa'_i}-f_{\aa_{i}}\in \fp$. Therefore the images of $f$ and $f'$ in $R$ are identical. In particular, $\Psi_d(f)=\Psi_d(f')\in \Bir(X)_d$.

Injectivity: Suppose that $f=(f_{\aa_0}:\cdots:f_{\aa_n})$ and $f'=(f_{\aa'_0}: \ldots: f_{\aa'_n})$ are two elements of $\tilde{H}_d$ such that $\Phi (\mathcal{D}'(f))=\Phi(\mathcal{D}'(f'))$.  Equivalently, $\Psi_d(f)=\Psi_d(f')$, or  $\psi_f=\psi_{f'}\in \Bir(X)_d$.
Since the base ideal of $\psi_f$ and $\psi_{f'}$, as ideals of $R$, have grade at least $2$ over $R$, they differ by multiplication into a constant. Thus, if  the difference depends on a constant multiplication,  lifting to the (quasi-)projective space  $\tilde{H}_d$, we have $(f_{\aa_0}:\ldots:f_{\aa_n})=(f_{\aa'_0}:\ldots:f_{\aa'_n})$--here is the place where we need to separate maps of clear degree from other maps.  The other point is that lifting from $\Bir(X)_d$ to  ${\tilde H}_d$ depends on  the natural ring map $S\to R$. Hence, the two representatives may differ by the equations of $\fp$, that is, (looking at the representatives as elements of $S$)  $f_{\aa_i}-f_{\aa_i'}\in \fp$ for all $i$.   Hence $\aa_i-\aa_i'\in V_{\fp}$. Therefore, $l_j(\aa_i-\aa_i')=0$ for all $i$ and $j$. It then follows that $$\mathcal{L}'(f_{\aa_0}-f_{\aa'_0}, \ldots, f_{\aa_n}-f_{\aa'_n})=0.$$
Since $\mathcal{D}$ is a linear map, we get $$\mathcal{D}'(f)=\mathcal{D}'(f').$$

Surjectivity: Let $h=(f_{\aa_0}:\cdots:f_{\aa_n}):X\dasharrow X$ be a birational map of clear  degree $d$. Since $\grade_R(f_{\aa_0},\ldots,f_{\aa_n})\geq 2$, we have $(f_{\aa_0}:\cdots:f_{\aa_n}) \in \theta(U_2)$. Since $h$ is dominant $\Dim(\Proj(k[f_{\aa_0},\ldots,f_{\aa_n}]))=\Dim(X)$ hence $\ell_R(f_{\aa_0},\ldots,f_{\aa_n})=r$; so that $(f_{\aa_0}:\cdots:f_{\aa_n})\in \theta( \mathcal{N}_{n+1})$, and finally since the birational map $h$ has an inverse of polynomial degree at most $d'=B(X,d)$ by  \autoref{Pinversedegree}, $(f_{\aa_0}:\cdots:f_{\aa_n}) \in \pi_2(Y)$; accordingly $(f_{\aa_0}:\cdots:f_{\aa_n})\in \pi_2(Y)\cap \theta(\mathcal{N}_{n+1}\cap U_2)=\tilde{H}_d$.  Therefore  $\Phi(\mathcal{D}'((f_{\aa_0}:\cdots:f_{\aa_n})))=h$.

\medskip
The bijectivity between $\Image(\mathcal{D}')$ and  $\Bir(X)_d$ ensues that the latter inherits the same topological structure as any one of the $\Image(\mathcal{D}')$. Since $\tilde{H}_d$ is quasi-projective,  by part (i), then  Chevalley's Theorem implies that $\Image(\mathcal{D}')$ is constructible,  and so is $\Bir(X)_d$.

(iii). According to the proof of part (ii), $\Bir(X)_{d}$ is embedded in the projective space where $\Image(\mathcal{D}')$ is embedded. Since $\mathcal{D}': \tilde{H}_d\longrightarrow \mathbb{P}^{(n+1)s-1}$, we have $\edim(\Bir(X)_{d})\leq (n+1)s-1$.  According to \autoref{s}, $s=N-{\HF}_{\fp}(d)$. The latter is equal to $\HF_{R}(d)$, since $N=\binomial{n+d}{d}=\HF_{S}(d)$.
\end{proof}


An immediate corollary is to decide when $\Bir(X)_d$ is indeed quasi-projective. 

\begin{Corollary}\label{C1} With the same token  as  \autoref{Tmain}, if $\HF_{\fp}(d)=0$, then $\Bir(X)_d$ has the structure of a quasi-projective variety. 
\end{Corollary}
\begin{proof} According to \autoref{s}, $\HF_{\fp}(d)=N-s$. Hence, in the course of the proof  of \autoref{Tmain}(ii),  the map $\mathcal{D}'$ is the identity map. So that $\Psi_{d}$ is a bijection. 
\end{proof}
We continue with another immediate corollary of \autoref{Tmain}(iii).
\begin{Corollary} \label{limsup}Let $d$ be an integer,  $k$  an infinite field and $X\subseteq \mathbb{P}^n$ an irreducible projective non-degenerated variety with coordinate ring $R$, Hilbert function $\HF_{R}$ and multiplicity $e(R)$. Then
\begin{equation}\label{limsup} \limsup\limits_{d\rightarrow \infty}\frac{\edim(\Bir(X)_{d})}{d^{\dim(X)}}\leq\frac{(n+1)e(R)}{\dim(X)!}.
\end{equation}
\end{Corollary}
The natural question arising by the above corollary is the following
\begin{Question} For a projective variety $X$ what is the big O value of function $\edim(\Bir(X)_{d})$. In other words, what is 
\begin{equation}\label{lims} \mathcal{O}(\edim(\Bir(X)_{d}):=\min\{m: \limsup\limits_{d\rightarrow \infty}\frac{\edim(\Bir(X)_{d})}{d^{m}}<\infty\}?
\end{equation}
\end{Question}
When $X$ is a curve then \autoref{X curve} shows that $\mathcal{O}(\edim(\Bir(X)_{d})=0$. The same holds when $X$ is a varity  of general type. 
For $X=\mathbb{P}^{2}$, Bisi et al. in  \cite{BCM15} show that $\dim(\Bir(X)_{d})=4d+6$. Hence in conjunction with \autoref{limsup}, we have  $\mathcal{O}(\edim(\Bir(\mathbb{P}^{2})_{d})\in\{1,2\}.$
\medskip

To follow the positive results of \cite{BF13} for an arbitrary variety $X$,  the next structural property,  \autoref{morphism}, is essential. To state this result, we recall the Demazure-Serre definition of the topology of $\Bir(X)$.
We slightly modify the original definition as follows. 

\begin{Definition} \label{maptoX} Let $X$ be an irreducible projective variety. Let $\mathfrak{A}$ be a projective constructible set. A {\it Demazure-Serre} morphism from $\mathfrak{A}$ to $\Bir(X)$ is a map 
$\zeta:\mathfrak{A}\to \Bir(X)$, $
a\mapsto \zeta_a,$
such that for any irreducible component $A$ of $\mathfrak{A}$, the induced map $
\mathfrak{F}_{\zeta}:A\times X \dasharrow A\times X,  (a,x)\mapsto (a,\zeta_a(x))$
is a birational map. 

A subset $F\subset \Bir(X)$ is closed if and only if for any  Demazure-Serre morphism $\zeta:\mathfrak{A}\to \Bir(X)$, the pre-image  $\zeta^{-1}(F)$ is closed in $\mathfrak{A}$.
\end{Definition}
Next theorem is a generalization of \cite[Lemma 2.4(c)]{BF13} for any irreducible variety $X\subseteq \mathbb{P}^n$. Here, due to the generality, we present  a different proof from  \cite{BF13} .

\begin{Proposition}\label{morphism} With the same token as in  \autoref{Tmain}, the map $\psi:{H}_d\to \Bir(X)$ defined in \autoref{psif} is a Demazure-Serre morphism. 
\end{Proposition}

\begin{proof} To prove that $\mathfrak{F}_{\psi}:{H}_d\times X\to {H}_d\times X$, $(f,x)\mapsto (x,\psi_{f(x)})$ is birational, we take an irreducible open affine subset $A\subseteq {H}_d$ and show that the map $\mathfrak{F}_{\psi}:A\times X\to A\times X$ is birational. To do this, we  first show that this map is dominant. 

 Let $\{y_{ij}:i=0,\ldots,n; j=1,\ldots,N\}$ be a set of variables defining the coordinate ring of $A$. Let $\yy_i=(y_{i1},\ldots,y_{iN})$ $i=0,\ldots,n$  and  $k[A]=k[\yy_0,\ldots,\yy_n]/\fa$ for some ideal $\fa$. For any $\aa\in A $, $\psi_{\aa}:X\dasharrow X$ is presented by $ f_{\aa}=(f_{\aa_0}:\cdots:f_{\aa_n})$ as defined in (\autoref{fa}).  Assume that the closure of the image of $\mathfrak{F}_{\psi}$ is contained in a closed subset of $A\times X$ defined by a polynomial $Q(\yy,\xx)$.  We have $Q(\yy,\xx)$  is homogeneous in terms of $\xx$-variables and  $Q(\aa,f_{\aa}(u))=0$ for any $u\in U$ for some open subset $U\subseteq X$. Since $X$ is irreducible, $Q(\aa,f_{\aa}(\xx))=0$. Since $f_{\aa}$ is birational there exists a polynomial $\alpha(\xx)$ with $f_{\aa}(f^{-1}_{\aa}(\xx))=\alpha(\xx)\xx$. Hence for some integer $m$, $\alpha(\xx)^mQ(\aa,\xx)=0$ for all $\aa$. Thus  $Q(\aa,\xx)=0$ for all $\aa$; yielding $Q(\yy,\xx)=0$. This proves that $\mathfrak{F}_{\psi}$ is dominant.

Since $\mathfrak{F}_{\psi}$ is a dominant map of varieties of the same dimension, it is a generically finite map. That means there exists an open subset $T\subset A\times X$ such that $$\mathfrak{F}_{\psi}:\mathfrak{F}_{\psi}^{-1}(T)\rightarrow T$$ is a finite morphism. We show that the degree of this finite morphism is $1$.

For each $\aa\in A$, there exist open subsetes $U_{\aa}\subseteq X$ and $V_{\aa}\subseteq X$ with $f_{\aa}:U_{\aa}\rightarrow V_{\aa}$ is an isomorphism. We claim that for some $\aa\in A$, the intersection $(\aa\times V_{\aa})\cap T$ is not empty. Since for any point in  $(\aa\times V_{\aa})\cap T$ the pre-image is a single set this proves that the degree of the above finite morphism is $1$. To prove the claim, let $P(\yy,\xx)$  be a polynomial homogeneous in $\xx$ variables that defines one of the equations of the complement of $T$. If by contrary, we suppose that for any $\aa\in A$ , $(\aa\times V_{\aa})\cap T=\emptyset$. Then for any $\aa\in A$ and $u\in U_{\aa}$, $P(\aa,f_{\aa}(u))=0$, using the birationality of $f_{\aa}$, we find a polynomial $\alpha(\xx)$ and an integer $m$ such that $\alpha(\xx)^mP(\aa,\xx)=0$ for all $\aa\in A$, accordingly $P(\yy,\xx)=0$ that implies   that  $T=\emptyset$ which is absurd. Hence there exists an $\aa\in A$ such that 
$(\aa\times V_{\aa})\cap T\neq \emptyset$.

\end{proof}

Now, having  \autoref{Tmain} and  \autoref{morphism}, one can follow the techniques in \cite[2.6--2.10]{BF13} to prove the following corollary.
\begin{Notation} We denote $\Image(\mathcal{D}')$ in  \autoref{Tmain}, by $\tilde{\tilde{H}}_d$, and the map $\Image(\mathcal{D}')\xrightarrow{\Phi }  \Bir(X)_d$ by $\Phi_d:\tilde{\tilde{H}}_d\to  \Bir(X)_d$.
\end{Notation}

\begin{Corollary}\label{CCC} With the same notation introduced before, during and after  \autoref{Tmain}, the following hold
\begin{enumerate}
\item{Let $A$ and $X$ be irreducible algebraic varieties and $\rho:A\to \Bir(X)$ be a Demazure-Serre morphism whose image is inside  $\Bir(X)_*$. There exists an open affine covering $\{A_i\}$ of $A$ such that for each $i$, there exists an integer $d_i$ and a morphism $\rho_i:A_i\to \tilde{\tilde{H}}_{d_i}$ such that the restriction of $\rho$ to $A_i$ is equal to $\Phi_{d_i}\circ \rho_i$.}
\item{A set $F\subseteq \Bir(X)_*$ is closed if and only if $\Phi_d^{-1}(F)$ is closed in $\tilde{\tilde{H}}_{d}$ for any $d$. }
\item{For any $d$, the set $\Bir(X)_{\leq d}$ is closed in $\Bir(X)_*$.}
\item{The topology of $\Bir(X)_*$ is the inductive limit topology of $\Bir(X)_{\leq d}$.}
\end{enumerate}
\end{Corollary}

\bibliographystyle{alpha}
\bibliography{ref}

\end{document}